\documentclass[reqno]{amsart}

\usepackage{xypic,hyperref,todonotes,comment,tikz}
\usetikzlibrary{arrows,calc,decorations.markings,math,arrows.meta}
\usepackage{graphics,amssymb,mathtools}
\usepackage{latexsym}
\usepackage{mathrsfs}
\xyoption{curve}
\usepackage{hyperref}
\hypersetup{ 
colorlinks,
citecolor=violet,
filecolor=violet,
linkcolor=black,
urlcolor=blue
}

\newtheorem{lemma}{Lemma}[section]
\newtheorem{proposition}[lemma]{Proposition}
\newtheorem{theorem}[lemma]{Theorem}
\newtheorem{corollary}[lemma]{Corollary}

\theoremstyle{definition}

\newtheorem{definition}[lemma]{Definition}
\newtheorem{remark}[lemma]{Remark}

\newcommand{\msc}[1]{\mathscr{{#1}}}
\newcommand{\mcl}[1]{\mathcal{{#1}}}
\newcommand{\mbb}[1]{\mathbb{#1}}
\newcommand{\mrm}[1]{\mathrm{#1}}
\newcommand{\mfk}[1]{\mathfrak{#1}}

\newcommand{\Prim}{\mathrm{Prim}}
\newcommand{\ot}{\otimes}

\newcommand{\End}{\mathrm{End}}

\renewcommand{\O}{\mathscr{O}}

\DeclareMathOperator{\Frac}{Frac}
\DeclareMathOperator{\YD}{YD}

\title[]{Pointed Hopf actions on central simple division algebras}
\date{\today}

\author{Pavel Etingof}
\email{etingof@mth.mit.edu}
\address{Department of Mathematics, Massachusetts Institute of Technology, Cambridge, MA
02139}

\author{Cris Negron}
\thanks{The author was supported by NSF Postdoctoral Research Fellowship DMS-1503147}
\email{negronc@mit.edu}
\address{Department of Mathematics, Massachusetts Institute of Technology, Cambridge, MA
02139}

\begin{document}
\maketitle

\begin{abstract}
We examine actions of finite-dimensional pointed Hopf algebras on central simple division algebras in characteristic $0$.  (By a Hopf action we mean a Hopf module algebra structure.)  In all examples considered, we show that the given Hopf algebra does admit a faithful action on a central simple division algebra, and we construct such a division algebra.  This is in contrast to earlier work of Etingof and Walton, in which it was shown that most pointed Hopf algebras do not admit faithful actions on fields.  We consider all bosonizations of Nichols algebras of finite Cartan type, small quantum groups, generalized Taft algebras with non-nilpotent skew primitive generators, and an example of non-Cartan type.
\end{abstract}

\section{Introduction}

This work is concerned with pointed Hopf actions on central simple division algebras, in characteristic $0$.  It is an open question~\cite[Question 1.1]{cuadraetingof16} whether or not an arbitrary finite-dimensional Hopf algebra can act inner faithfully on such a division algebra.  A conjecture of Artamonov also proposes that any finite-dimensional Hopf algebra should act inner faithfully on the ring of fractions of a quantum torus~\cite[Conjecture 0.1]{artamonov05}, and it is known that the parameters appearing in such a quantum torus cannot (all) be generic~\cite[Theorem 1.8]{etingofwalton16}.  
\par

We focus here on examples, and consider exclusively pointed Hopf algebras with abelian group of grouplikes.  Such algebras are well-understood via the extensive work of many authors, e.g.~\cite{heckenberger06,andruskiewitschschneider10,angiono13}.

\begin{theorem}\label{thm:01}
The following Hopf algebras admit an inner faithful Hopf action on a central simple division algebra:
\begin{itemize}
\item Any bosonization $H=B(V)\rtimes G$ of a Nichols algebra of a finite Cartan type braided vector space via an abelian group $G$ (as defined in~\cite{andruskiewitschschneider10}).
\item The small quantum group $u_q(\mfk{g})$ of a semisimple Lie algebra $\mfk{g}$.
\item Generalized small quantum groups $u(\mcl{D})$ such that the space of skew primitives in $u(\mcl{D})$ generate $\operatorname{Rep}(G)$ (as a tensor category), where $G$ is the group of grouplikes in $u(\mcl{D})$.
\item Generalized Taft algebras $T(n,m,\alpha)$, where $m\mid n$ and $\alpha\in \mbb{C}$.
\item The $64$-dimensional Hopf algebra $H=B(W)\rtimes \mbb{Z}/4\mbb{Z}$, where $W$ is the $2$-dimensional braided vector space with braiding matrix {\small$\left[\begin{array}{cc}
-1 & i\\
-1 & i
\end{array}\right]$}.
\end{itemize}
\end{theorem}

In each of the examples appearing in Theorem~\ref{thm:01}, an explicit central simple division algebra with an inner faithful action is constructed.  We also consider in each case whether the action we construct is Hopf-Galois.
\par

As mentioned in the abstract, our results contrast with those of Etingof-Walton \cite{etingofwalton15,etingofwaltonII}.  In~\cite{etingofwalton15} the authors show that any generalized Taft algebra $T(n,m,\alpha)$ which admits an inner faithful action on a field is a standard Taft algebra $T(m,m,0)$.  Although more general Cartan type algebras $B(V)\rtimes G$ are not directly considered in~\cite{etingofwalton15,etingofwaltonII}, this restriction on Taft actions already obstructs actions of general bosonizations $B(V)\rtimes G$, as each pair $(g,v)$ of a grouplike $g\in G$ and $(g,1)$-skew primitive $v\in V$ generates a generalized Taft algebra in $B(V)\rtimes G$.  Similarly, small quantum groups outside of type $A_1$ were shown to not act inner faithfully on fields in~\cite{etingofwalton15,etingofwaltonII}.
\par

Our methods are based on the observation that, for $H$ a pointed Hopf algebra with abelian group of grouplikes $G$, and $Q$ a central simple division algebra with an $H$-action, the skew primitives in $H$ must act as inner skew derivations on $Q$ (see Theorem~\ref{thm:TQstructure} and Lemma~\ref{lem:inn_der} below).  Hence actions of $H$ on a given $Q$ are parametrized by a choice of a grading by the character group of $G$, and a corresponding choice of a collection of elements in $Q$ which solve certain universal equations for (the skew primitives in) $H$.
\par

The universal approach to Hopf actions we have just described is discussed in more detail, at least in the case of coradically graded $H$, in Section~\ref{sect:univ}.

\subsection*{Acknowledgements}

We thank Iv\'{a}n Angiono, Juan Cuadra, and Chelsea Walton for helpful conversations.  The first author was partially supported by the US National Science Foundation grant DMS-1502244.  The second author was supported by the US National Science Foundation Postdoctoral Research Fellowship DMS-1503147.

\tableofcontents

\section{Preliminaries}

\subsection{Conventions}

All algebras, vector spaces, etc.\ are over $\mbb{C}$.  For a Hopf algebra $H$ we let $G(H)$ denote the group of grouplike elements.  Given a Hopf algebra $H$ and a grouplike $g\in G(H)$ we let $\Prim_g(H)$ denote the $\mbb{C}$-subspace of $(g,1)$-skew primitives.  We take
\[
\Prim(H)=\oplus_{g\in G}\Prim_g(H).
\]

Given a finite-dimensional Hopf algebra $H$ and $H$-module algebra $A$, we say that $A$ is $H$-Galois over its invariants $A^H$ if, under the corresponding $H^\ast$-coaction, $A$ is an $H^\ast$-Galois extension of its coinvariants $A^H=A^{\mrm{co}H^\ast}$.

\subsection{The category $\YD(G)$}

We recall some standard notions, which can be found in~\cite{andruskiewitschschneider10} for example.  The category of Yetter-Drinfeld modules over a group $G$ is the category of simultaneous left $G$-representations and left $kG$-comodules $V$ which satisfy the compatibility
\[
\rho(g\cdot v)=(gv_{-1}g^{-1})\ot gv_0,
\]
where $g\in G$, $v\in V$, and $\rho(v)=v_{-1}\ot v_0$ denotes the $kG$-coaction.  This category is braided, with braiding
\[
c_{V,W}:V\ot W\to W\ot V,\ \ v\ot w\mapsto (v_{-1}w)\ot v_0.
\]
We will focus mainly on Yetter-Drinfeld modules over abelian $G$, in which case the action and coaction simply commute.
\par

For algebras $A$ and $B$ in $\YD(G)$, we define the braided tensor product $A\underline{\ot}B$ as the vector space $A\ot B$ with product
\[
(a\ot b)\cdot (a'\ot b')=\big(a(b_{-1}a')\big)\ot \big(b_0b'\big).
\]
The object $A\underline{\ot}B$ is another algebra in $\YD(G)$ under the diagonal action and coaction.  We can also define the braided opposite algebra $A^{\underline{op}}$, which is the vector space $A$ with multiplication $a\cdot_{\underline{op}}b=(a_{-1}b)a_0$.

A Hopf algebra in $\YD(G)$ is an algebra $R$ in $\YD(G)$ equipped with a coalgebra structure such that the comultiplication $\Delta_R:R\to R\underline{\ot}R$ is a map of algebras in $\YD(G)$.  Such an $R$ should also come equipped with an antipode $S_R:R\to R$ which is a braided anti-algebra and anti-coalgebra map satisfying $S_R(r_1)r_2=r_1S_R(r_2)=\epsilon(r)$, for each $r\in R$.

\begin{definition}
Given a Hopf algebra $R$ in $\YD(G)$, the bosonization of $R$ is the smash product algebra $R\rtimes G$.
\end{definition}

Any bosonization $R\rtimes G$ is well-known to be a Hopf algebra with unique Hopf structure $(\Delta,\epsilon, S)$ such that $k[G]$ is a Hopf subalgebra, and on $R\subset R\rtimes G$ we have
\[
\Delta(r)=r_1(r_2)_{-1}\ot (r_2)_0,\ \ \epsilon(r)=\epsilon_R(r),\ \ S(r)=S_{k[G]}(r_{-1})S_R(r).
\]
The bosonization operation is also referred to as the Radford biproduct in the literature.

\begin{lemma}\label{lem:YD_act}
Let $A$ be an algebra in $\YD(G)$.  Suppose $R$ acts on $A$ in such a way that the action map $R\ot A\to A$ is a morphism in $\YD(G)$ and 
\[
r\cdot (ab)=\big(r_1(r_2)_{-1}a\big)\big((r_2)_0b\big)
\]
for $r\in R$, $a,b\in A$.  Then $A$ is a module algebra over the bosonization $R\rtimes G$, where $G$ acts on $A$ via the Yetter-Drinfeld structure and the $R$-action is unchanged.
\end{lemma}

\begin{proof}
This is immediate from the definition of the comultiplication on the bosonization.
\end{proof}

\subsection{Hopf actions on division algebras}

Recall that for a domain $A$ which is finite over its center, we have the division algebra $\Frac(A)$, which one can construct as the localization via the center $\Frac(A)=\Frac(Z(A))\ot_{Z(A)}A$.

\begin{theorem}[{\cite[Theorem 2.2]{skryabinoystaeyen06}}]\label{thm:indQ}
Suppose a Hopf algebra $H$ acts on a domain $A$ which is finite over its center.  Then there is a unique extension of this $H$-action to an action on the fraction division algebra $\Frac(A)$.
\end{theorem}

\begin{remark}
The result from~\cite{skryabinoystaeyen06} is significantly more general than what we have written here.  They show that an $H$-action extends to $\Frac(A)$, essentially, whenever a reasonable algebra of fractions exists for $A$ (with no reference to the center).
\end{remark}

When considering actions on division algebras, one can assess the Hopf-Galois property for the extension $Q^H\to Q$ via a rank calculation.

\begin{theorem}[{\cite[Theorem 3.3]{cfm90}}]\label{thm:Gal}
Suppose a finite-dimensional Hopf algebra $H$ acts on a division algebra $Q$.  Then $Q$ is $H$-Galois over $Q^H$ if and only if $\mrm{rank}_{Q^H}Q=\mrm{dim}H$.
\end{theorem}

\subsection{Faithfulness of pointed Hopf actions}

Recall that $\Prim_g(H)$ denotes the subspace of $(g,1)$-skew primitives in a Hopf algebra $H$, for $g$ an arbitrary grouplike.  Take $\Prim_g(H)'$ to be the sum of all the nontrivial eigenspaces for $\Prim_g(H)$ under the adjoint action of $g$.
\par

For finite-dimensional pointed $H$, we have that the nilpotence order of any $g$-eigenvector $x$ in the degree $1$ portion $\Prim_g(\mrm{gr} H)_1$ is less than or equal to the order of the associated eigenvalue.  So we see that the map
\[
\Prim_g(H)'\to \Prim_g(H)/\mbb{C}(1-g)=\Prim_g(\mrm{gr} H)_1
\]
is an isomorphism.  Now by the Taft-Wilson decomposition of the first portion of the coradical filtration $F_1H$~\cite{taftwilson74}, we have
\begin{equation}\label{eq:97}
F_1H=\mbb{C}[G]\oplus\left(\bigoplus_{g,h\in G}h\cdot \Prim_g(H)'\right),
\end{equation}
where $G=G(H)$.

\begin{lemma}\label{lem:fate}
Let $H$ be a finite-dimensional pointed Hopf algebra, and $A$ be an $H$-module algebra.  Suppose that the $G(H)$ action on $A$ is faithful, and that for each $g\in G(H)$ the map $\Prim_g(H)'\to \End_k(A)$ is injective.  Then the $H$-action on $A$ is inner faithful.
\end{lemma}

\begin{proof}
Take $G=G(H)$.  Suppose we have a factorization $H\to K\to \End_k(A)$, where $\pi:H\to K$ is a Hopf projection.  By considering the dual inclusion $K^\ast\to H^\ast$ we find that $K$ is pointed as well.  By faithfulness of the $G$-action we have that $\pi|_{G}$ is injective.  Furthermore, each $\pi|_{\Prim_g(H)'}$ is injective by hypothesis, and each $\Prim_g(H)'$ maps to $\Prim_g(K)'$.  By the decomposition~\eqref{eq:97}, where we replace $H$ with $K$, we find that the restriction $F_1H\to F_1K$ is injective.  It follows that $\pi$ is injective~\cite[Prop. 2.4.2]{heynemanradford74}, and therefore an isomorphism.
\end{proof}

In the case in which the group of grouplikes $G=G(H)$ is abelian, the entire group $G$ acts on each $\Prim_g(H)$, and we can decompose the sum of the primitive spaces $\Prim(H)$ as
\[
\mbb{C}1_H\oplus \Prim(H)=\mbb{C}[G]\oplus \Prim(H)',
\]
where $\Prim(H)'$ is the sum of the nontrivial eigenspaces.

\begin{corollary}\label{cor:fate}
Suppose $H$ is finite-dimensional and pointed, with abelian group of grouplikes.  Then an action of $H$ on an algebra $A$ is inner faithful provided $G(H)$ acts faithfully on $A$ and the restriction of the representation $H\to \End_k(A)$ to $\Prim(H)'$ is injective.
\end{corollary}

\begin{proof}
We have $\Prim(H)'=\oplus_g\Prim_g(H)'$ in this case.
\end{proof}

\section{Actions of generalized Taft algebras} 

We consider for positive integers $m\leq n$, with $m\mid n$, the Hopf algebra
\[
T(n,m,\alpha)=\frac{\mbb{C}\langle x,g\rangle}{(x^m-\alpha(1-g^m), g^n-1, gxg^{-1}-qx)},
\]
where $q$ is a primitive $m$-th root of $1$.  In the algebra $T(n,m,\alpha)$ the element $g$ is grouplike and $x$ is $(g,1)$-skew primitive.
\par

We apply Theorem~\ref{thm:TQstructure} below to obtain actions of these Hopf algebras on central simple division algebras.  At $\alpha=0$, the division algebra we produce is the ring of fractions of a quantum plane, while the division algebra we produce for $T(n,m,1)$ has a more intricate structure.

\subsection{Generic actions of pointed Hopf algebras and Taft algebras}
Let us take a moment to consider actions of pointed Hopf algebras in general, before returning to the specific case of generalized Taft algebras.

We note that for a pointed Hopf algebra $H$ each skew primitive $x_i$ determines a Hopf embedding $T(n_i,m_i,\alpha_i)\to H$.  An action of $H$ on an algebra $A$ is then determined by an action of the group $G(H)$ and compatible actions of the Hopf subalgebras $T(n_i,m_i,\alpha_i)\to H$.  Whence we study actions of the generalized Taft algebras $T(n,m,\alpha)$ in order to understand actions of pointed Hopf algebras more generally.

The following result motivates most of our constructions, even when it is not explicitly referenced.  The proof is non-trivial and is given in Section~\ref{sect:proof}.

\begin{theorem}\label{thm:TQstructure}
Suppose $T(n,m,\alpha)$ acts on a central simple algebra $A$, and fix $\zeta$ a primitive $n$-th root of $1$ with $\zeta^{\frac{n}{m}}=q$.  Let $A=\oplus_{i=0}^n A_i$ be the corresponding decomposition of $A$ into eigenspaces, so that $g$ acts as $\zeta^i$ on $A_i$.  Then there exists $c\in A_{n/m}$ such that $x\cdot a=ca-\zeta^{|a|}ac$ for each (homogeneous) $a\in A$.  Furthermore, this element $c$ satisfies the commutativity relation
\begin{equation}\label{eq:comm_rel}
c^ma-\zeta^{m|a|}ac^m=\alpha(1-\zeta^{m|a|})a
\end{equation}
for each homogeneous $a\in A$.
\par

Conversely, if $A=\oplus_{i=0}^n A_i$ is a $\mbb{Z}/n\mbb{Z}$-graded central simple division algebra, and $c\in A_{n/m}$ is such that $c^ma-\zeta^{m|a|}ac^m=\alpha(\zeta^{m|a|}-1)a$ for each homogeneous $a\in A$, then there is a (unique) action of the generalized Taft algebra $T(n,m,\alpha)$ on $A$ given by
\[
g\cdot a=\zeta^{|a|}a\ \ \text{and}\ \ x\cdot a=ca-\zeta^{|a|}ac
\]
which gives $A$ the structure of a $T(n,m,\alpha)$-module algebra.
\end{theorem}

Now, for general $H$ with abelian group of grouplikes, if $H$ acts on a central simple algebra $A$ then we decompose $A$ into character spaces $A=\oplus_\mu A_\mu$ for the action of $G$.  For each skew homogeneous $(g_i,1)$-skew primitive $x_i\in H$, with associated character $\chi_i$, we have the generalized Taft subalgebra $T(n_i,m_i,\alpha_i)\to H$.  By restricting the action, and considering Theorem~\ref{thm:TQstructure}, we see that each $x_i$ acts on $A$ as an operator
\[
x_i\cdot a=c_ia-\mu(g_i)ac_i,\ \ \text{for }a\in A_\mu,
\]
for an element $c_i\in A_{\chi_i}$.  Whence the action of $H$ is determined by a choice of a $G^\vee$-grading on $A$ and a choice of elements $c_i\in A_{\chi_i}$ satisfying relations~\eqref{eq:comm_rel} (as well as all other relations for $H$).  We return to this topic in Sections~\ref{sect:proof} and~\ref{sect:univ}.

\subsection{A Hopf-Galois action for generalized Taft algebras at $\alpha=0$}

Consider $T(n,m,0)$ as above, with $q$ a primitive $m$-th root of $1$.  It was shown in~\cite{etingofwalton15} that this algebra admits no inner faithful action on a field when $n>m$.
\par

Take $K=\mbb{C}(u,v)$ and consider the cyclic algebra
\[
Q(n,m)=Q_\zeta(n,m):=K\langle c,w\rangle/(c^n-u,w^n-v,cw- \zeta wc),
\]
where $\zeta$ is a chosen primitive $n$-th root of $1$ with $\zeta^{\frac{n}{m}}=q$.  The algebra $Q(n,m)$ is a cyclic division algebra of degree $n$ over $K$.

\begin{proposition}\label{prop:gen_taft0}
The central simple division algebra $Q(n,m)$ admits an inner faithful $T(n,m,0)$-action which is uniquely specified by the values
\[
g\cdot c=qc,\ \ g\cdot w=\zeta w,\ \ x\cdot c=(1-q)c^2,\ \ x\cdot w=0.
\]
Furthermore, $Q(n,m)$ is $T(n,m,0)$-Galois over its invariants $Q(n,m)^{T(n,m,0)}$.
\end{proposition}

\begin{proof}
Take $s=\frac{n}{m}$. The existence of the proposed inner faithful action follows by Theorem~\ref{thm:TQstructure}.  So we need only address the Hopf-Galois property.  Take $T=T(n,m,0)$ and define $[c,a]_\mrm{sk}:=ca-(g\cdot a)c$ for arbitrary $a\in Q(n,m)$.
\par

As for the Hopf-Galois property, we consider the basis of monomials $\{c^iw^j\}_{i,j=0}^{n-1}$ for $Q(n,m)$, considered as a vector space over the field $K=\mbb{C}(u,v)=\mbb{C}(c^{n},w^{n})$. The elements $c^m$ and $w^{n}$ are both $g$-invariant and 
\[
\mrm{ad}_\mrm{sk}(c)(c^m)=[c,c^m]=0,\ \ \mrm{ad}_\mrm{sk}(c)(w^{n})=[c,w^n]=0.
\]
So the degree $m$ field extension $K(c^m)\subset Q(n,m)$ lies in the $T$-invariants.  The algebra $Q(n,m)$ is free over $K(c^m)$ on the left with basis 
\[
\{c^iw^j:0\leq i<m,\ 0\leq j<n\}.
\]

Now, for a generic element
\[
f=\sum_{0\leq i< m,\ 0\leq j<n}f(i,j)c^iw^j\ \in\ Q(n,m)
\]
with the coefficients $f(i,j)\in K(c^m)$ we have $g\cdot f=\sum_{i,j}\zeta^{si+j}f(i,j)c^iw^j$.  So for $g$-invariant $f$ we have $f=\sum_{i=0}^{m-1}f(i)c^iw^{s(m-i)}$.  Now applying $x$ gives
\[
x\cdot f=\sum_{i=0}^{m-1}(1-q^i)f(i)c^{i+1}w^{s(m-i)}.
\]
So $x\cdot f=0$ requires $f=f(0)$.  This identifies the invariants $Q(n,m)^T$ with the subfield $K(c^m)$.  Hence $Q(n,m)$ is free of rank $mn=\mrm{dim}T$ over its invariants, and we find that $Q(n,m)$ is $T$-Galois.
\end{proof}

\subsection{An action for generalized Taft algebras at non-zero parameter $\alpha$}

By rescaling the skew primitive, we have a Hopf isomorphism $T(n,m,\alpha)\cong T(n,m,1)$ whenever $\alpha$ is nonzero.  We wish to produce a central simple algebra and corresponding action for $T(n,m,1)$.
\par

Take $K=\mbb{C}(w)$ and consider the polynomial $p_{n,m}(X)=(X^m-1)^{\frac{n}{m}}-w$ over $K$.  Take $s=n/m$ and $\zeta$ a primitive $n$-th root of $1$ with $\zeta^s=q$.
\par

We let $L$ denote the splitting field of $p_{n,m}$ over $K$.  The field $L$ is generated, over $K$, by a choice of $s$-th root $\sqrt[s]{w}$ for $w\in K$ and solutions $c_j$ to the equation $X^m-\zeta^{jm}\sqrt[s]{w}-1=0$, for $1\leq j\leq s$.
\par

We note that scalings of the $c_k$ by $m$-th roots of unity provide all $n$ (distinct) roots to our equation $p_{n,m}\in K[X]$.  Consider the automorphisms $g_i$ and $\sigma$ of $L$ over $K$ defined by $g_i(c_j)=q^{\delta_{ij}}c_j$ and $\sigma(c_j)=c_{j+1}$.  (We abuse notation so that $c_{s+1}=c_1$.)  By comparing the degree of $L$ over $K$ with the order of the subgroup of $\operatorname{Aut}_K(L)$ generated by the $g_i$ and $\sigma$, one finds that the extension $L/K$ is Galois with Galois group
\[
\mrm{Gal}(L/K)=\langle g_i:1\leq i\leq s\rangle\rtimes \langle \sigma\rangle\cong (\mbb{Z}/m\mbb{Z})^s\rtimes \mbb{Z}/s\mbb{Z}.
\]
\par

We consider the Ore extension $L[t;\sigma]$.  This algebra is a domain which is finite over its center, and we take
\[
Q=\Frac(L[t;\sigma]).
\]
We produce below an action of $T(n,m,1)$ on $Q$.
\par

We first extend the automorphism $g|_L=\prod_{i=1}^sg_i:L\to L$, $c_i\mapsto qc_i$, to an automorphism $g:Q\to Q$ such that $g(t)=\zeta t$.  We note that such an extension is well-defined since $(g|_L)\sigma=\sigma(g|_L)$.  The automorphism $g$ is order $n$, and we obtain an action of $\mbb{Z}/n\mbb{Z}=G(T(n,m,1))$ on $Q$.

\begin{lemma}
Take $Q$ as above, with the given $\mbb{Z}/n\mbb{Z}$-action.  Then, at arbitrary $a\in Q$, each element $c_i\in Q$ satisfies
\[
c_i^ma-(g^m\cdot a)c_i^m=(1-g^m)\cdot a.
\]
\end{lemma}

\begin{proof}
Take $\zeta$ an $s$-th root of $q$ as above.  It suffices to provide the relation on $L[t;\sigma]$.  Any homogeneous element of $L[t;\sigma]$ may be written in the form $bt^r$, with $b\in L$.  Note that $c_i^m-1=\tau w^{1/s}$ for each $i$, where $\tau$ is a root of unity, and $\sigma(w^{1/s})=\zeta^mw^{1/s}$.  Note also that $g^m|_L=id_L$.  We therefore have
\[
\begin{array}{rl}
\tau^{-1}\left((c_i^m-1)bt^r-(g^m\cdot bt^r)(c_i^m-1)\right) & = w^{1/s}bt^r- b(g^m\cdot t^r)w^{1/s}\\
&=w^{1/s}bt^r- \zeta^{mr}bt^rw^{1/s}\\
&=bt^r\sigma^r(w^{1/s})- \zeta^{mr}bt^rw^{1/s}\\
&=0.
\end{array}
\]
Thus $(c_i^m-1)y-(g^m\cdot y)(c_i^m-1)=0$ for all $y\in L[t;\sigma]$.  The fact that $(c_i^m-1)$ commutes with $1=yy^{-1}$ implies that $(c_i^m-1)$ satisfies the same relation for all $a$ in the ring of fractions $Q$.  We rearrange to arrive at the desired equation.
\end{proof}

\begin{proposition}\label{prop:gen_taft1}
For any non-zero $\alpha\in \mbb{C}$, there is an inner faithful $T(n,m,\alpha)$-action on the central simple division algebra $Q=\Frac(L[t;\sigma])$.  This action is not Hopf-Galois.
\end{proposition}

\begin{proof}
We may assume $\alpha=1$.  Take $s=n/m$, $G=G(T(n,m,1))=\langle g\rangle$, and let $\zeta$ be the give primitive $n$-th root of unity with $\zeta^s=q$.  We provide a $G$-action on $Q$ by letting $g$ act as the above automorphism $g(c_i)=qc_i$, $g(t)=\zeta t$.  If we grade $Q$ as $Q=\oplus_{i=0}^{n-1}Q_i$, with $g|_{Q_i}=\zeta^i\cdot-$, then $c_i\in Q_s$, and any choice $c=c_i$ provides an element which satisfies the equation
\[
c^ma-(g^m\cdot a)c^m=(1-g^m)\cdot a
\]
at each $a\in Q$.  We therefore apply Theorem~\ref{thm:TQstructure} to arrive at an explicit action of $T(n,m,1)$ on $Q$.
\par

As for inner faithfulness, the fact that $G$ acts faithfully on $Q$ is clear, and the fact that $\mrm{ad}_\mrm{sk}(c)\neq 0$ follows from the fact that $\mrm{ad}_\mrm{sk}(c)(c)=(1-q)c^2\neq 0$.  Thus the action of $T(n,m,1)$ is inner faithful by Corollary~\ref{cor:fate}.
\par

As for the Hopf-Galois property, we consider the invariants $L[t;\sigma]^G$ and decompose $L=\oplus_{k=0}^{m-1}L_{ks}$, with $g|_{L_{ks}}=q^k\cdot-$.  Then $L=L_0[\alpha]$, for arbitrary nonzero $\alpha\in L_{-s}$, and one calculates that the invariants is a polynomial ring $L[t;\sigma]^G=L_0[\alpha t^s]$.  Now we have
\[
L[t;\sigma]=L_0[\alpha t^s]\cdot (\oplus_{j=0}^{s-1}Lt^j)=L_0[\alpha t^s]\cdot\{\alpha^at^b:0\leq a<m,\ 0\leq b<s\},
\]
from which one can conclude
\[
\mrm{rank}_{L[t;\sigma]^G}L[t;\sigma]=sm.
\]
Since $\sigma$ is order $s$, we have 
\[
L[t;\sigma]^G=L_0[\alpha t^s]\subset Z(L[t;\sigma]),
\]
and $\mrm{ad}_\mrm{sk}(c)|_{L[t;\sigma]^G}=0$.  Hence the $G$-invariants in $L[t;\sigma]$ is the entire $T(n,m,1)$-invariants.  We may write the fraction field as the localization
\[
Q=\Frac(L[t;\sigma])=\Frac(L[t;\sigma]^G)\ot_{L[t;\sigma]^G}L[t;\sigma]
\]
to find that $Q^T=Q^G=\Frac(L[t;\sigma]^G)$ and
\[
\dim_{Q^T}Q= \dim_{Q^G}Q=sm<nm=\dim T(n,m,1).
\]
Hence the action is not Hopf-Galois, by Theorem~\ref{thm:Gal}.
\end{proof}

\section{Actions of graded finite Cartan type algebras}

We consider a class of pointed Hopf algebras which generalize small quantum Borel algebras.  These are pointed, coradically graded, Hopf algebras of finite Cartan type.  We first recall the construction of these algebras, then provide corresponding central simple division algebras on which these Cartan type algebras act inner faithfully.

\subsection{Cartan type algebras (following~\cite{andruskiewitschschneider10})}

Let $V=\mbb{C}\{x_1,\dots,x_\theta\}$ be a braided vector space of diagonal type, with braiding matrix $[q_{ij}]$.  Rather, the coefficients $q_{ij}$ are such that $\mfk{c}_{V,V}(x_i\ot x_j)=q_{ij}x_j\ot x_i$, where $\mfk{c}_{V,V}$ is the braiding on $V$.  We assume that the $q_{ij}$ are roots of unity so that $V\in \YD(G)$ for a finite abelian group $G$.
\par

Following Andruskiewitsch and Schneider, we say $V$ is of {\it Cartan type} if there is an integer matrix $[a_{ij}]$ such that the coefficient $q_{ij}$ satisfy
\begin{equation}\label{eq:180}
q_{ij}q_{ji}=q_{ii}^{a_{ij}}.
\end{equation}
We always suppose $a_{ii}=2$ and $0\leq -a_{ij}<\mrm{ord}(q_{ii})$ for distinct indices $i,j$.  We say $V$ is of {\it finite Cartan type} if the associated Nichols algebra $B(V)$ is finite-dimensional.  We have the following fundamental result of Heckenberger.

\begin{theorem}[{\cite[Theorem 1]{heckenberger06}}]
Suppose $V$ is of Cartan type.  Then the Nichols algebra $B(V)$ is finite-dimensional if and only if the associated matrix $[a_{ij}]$ is of finite type, i.e. if and only if $[a_{ij}]$ is the Cartan matrix associated to a semisimple Lie algebra over $\mbb{C}$ up to permutation of the indices.
\end{theorem}

Consider $V$ of finite Cartan type, we have the associated root system $\Phi$, with basis $\{\alpha_i\}_i$ indexed by a homogeneous basis for $V$.  Let $\Gamma$ be the associated union of Dynkin diagrams.  We decompose $\Phi$ into irreducible components
\[
\Phi={\coprod}_{I\in\pi_0(\Gamma)}\Phi_I.
\]
Throughout we assume the following two additional restrictions:
\begin{itemize}\label{item:ord}
\item $q_{ii}$ is of odd order.
\item $q_{ii}$ is of order coprime to $3$ when the associated component $\Phi_I$, with $\alpha_i\in I$, is of type $G_2$.
\end{itemize}

By~\cite[Lemma 2.3]{andruskiewitschschneider10} we have that $N_i=\mrm{ord}(q_{ii})$ is constant for all $i$ with associated simple roots $\alpha_i$ in a given component of the Dynkin diagram.  For $\gamma\in \Phi^+_I$ we take $N_\gamma=N_i$ for any $i$ in component $I$.

For finite Cartan type $V$ and $\gamma\in \Phi^+$ one has associated root vectors $x_\alpha$, which are constructed via iterated braided commutators as in~\cite{andruskiewitschschneider02,lusztig_book}.

\begin{theorem}[{\cite[Theorem 5.1]{andruskiewitschschneider10}}]\label{thm:present}
Suppose $R=\mfk{B}(V)$ is of Cartan type, and take $N_i=\mrm{ord}(q_{ii})$.  Then $R$ admits a presentation $R=TV/I$, where $I$ is generated by the relations
\begin{itemize}
\item {\rm (Nilpotence relations)} $x_\gamma^{N_\alpha}$ for $\gamma\in \Phi^+$;
\item {\rm ($q$-Serre relations)} $\mrm{ad}_\mrm{sk}(x_i)^{1-a_{ij}}(x_j)$;
\end{itemize}
\end{theorem}

\subsection{Actions of finite Cartan type algebras}

We call a Hopf algebra $H$ of {\it (finite) Cartan type} if $H=B(V)\rtimes G$ for $V$ of (finite) Cartan type and $G$ a finite abelian group.  For a $G\times G^\vee$-homogeneous basis vectors $x_i\in V$ we write $g_i$ for the group element associated to $x_i$, $\Delta_H(x_i)=x_i\ot 1+g\ot x_i$, and $\chi_i$ for the associated character $\mrm{Ad}_g(x_i)=\chi_i(g)x_i$.

\begin{theorem}\label{thm:cart_type}
Take $H=B(V)\rtimes G$ of finite Cartan type, and let $[q_{ij}]$ be the braiding matrix for $V=\mbb{C}\{x_1,\dots, x_\theta\}$.  Let $[a_{ij}]$ be the matrix encoding the relations~\eqref{eq:180}, and suppose that the $x_i$ are ordered so that $[a_{ij}]$ is block diagonal with each block a standard Cartan matrix associated to a Dynkin diagram.  Then for any subset $Y=\{\mu_1,\dots, \mu_t\}\subset G^\vee$ there is an $H$-action on the algebra
\[
A(Y)=\frac{\mbb{C}\langle c_1,\dots, c_\theta,w_1,\dots, w_t\rangle}{(c_ic_j-q_{ij}c_jc_i,\ c_kw_m-\mu_m(g_k)w_mc_k: i<j)}
\]
and on the central simple division algebra $Q(Y)=\Frac(A(Y))$.  This action is uniquely specified by the values on the generators
\[
g\cdot c_i=\chi_i(g)c_i,\ \ x_j\cdot c_i=c_jc_i-q_{ji} c_jc_i,\ \ g\cdot w_k=\mu_k(g)w_k,\ \ x_l\cdot w_k=0,
\]
and is inner faithful if and only if the subset $\{\chi_i\}_{i=1}^\theta\cup Y$ generates $G^\vee$.
\end{theorem}

The proof of Theorem~\ref{thm:cart_type} is given in Section~\ref{sect:proofcartanQ}.  The main difficulty in producing such an action is showing that the proposed action does in fact satisfy the relations of $H$.

\begin{comment}
\begin{remark}
The condition $2\nmid \mrm{ord}(q_{ii})$ is somewhat artificial here.  The issue (for us) is largely notational.  When $2$ divides $N=\mrm{ord}(q_{ii})$, the element $x_i^{N/2}$ will be skew primitive, and the coradical grading on the Hopf algebra with relations by Theorem~\ref{thm:present} will not agree with the grading by $x_i$-degree.  However, if we want, we can instead kill the elements $x_i^{N/2}$ in this case.  So, once one decides what Hopf algebra they would like to consider in this case, the methods of proof for Theorem~\ref{thm:cart_type} will generally apply.
\end{remark}
\end{comment}

We note that the algebra $Q(Y)$ is not $H$-Galois outside of type $A_1$.  This follows by a rank calculation which we do not repeat here.  In type $A_1$ we have produced a Hopf-Galois action already in Proposition~\ref{prop:gen_taft0}.

\subsection{The pre-Nichols algebra}

Let $G$ be a finite abelian group.  Take $V$ in $\YD(G)$ of finite Cartan type, and fix $R=B(V)$.  Consider a basis $\{x_1,\dots,x_\theta\}$ for $V$, with each $x_i$ homogeneous with respect to the $G\times G^\vee$-grading.  We take $g_i=\deg_G(x_i)$ and $\chi_i=\deg_{G^\vee}(x_i)$.
\par

Let $[q_{ij}]$ be the braiding matrix for $V$.  We assume the orders $\mrm{ord}(q_{ii})$ are odd, and additionally that $\mrm{ord}(q_{ii})$ is coprime to $3$ in type $G_2$.  We recall here some work of Andruskiewitsch and Schneider.

\begin{theorem}[\cite{andruskiewitschschneider10}]
For $R=B(V)$ of finite Cartan type, the algebra
\[
R^\mrm{pre}:=TV/(q\text{-}\mrm{Serre\ relations})
\]
is a Hopf algebra in $\YD(G)$, with Hopf structure induced by the quotient $TV\to R^\mrm{pre}$.
\end{theorem}

We refer to $R^\mrm{pre}$ as the distinguished pre-Nichols algebra associated to $R$, following Angiono~\cite{angiono16}.  For $H=R\rtimes G$ we call $H^\mrm{pre}:=R^\mrm{pre}\rtimes G$ the ADK form of $H$, in reference to Angiono, de Concini, and Kac.
\par

As with the usual de Concini-Kac algebra, there is an action of the braid group of $R^\mrm{pre}$ which gives us elements $x_\gamma=T_\sigma(x_i)$ as in~\cite{andruskiewitschschneider02,lusztig_book}.

\begin{theorem}[{\cite[Theorem 2.6]{andruskiewitschschneider10}}]
Let $Z_0$ be the subalgebra of $R^\mrm{pre}$ generated by the powers $x_\gamma^{N_\gamma}$.  The subalgebra $Z_0$ is a Hopf subalgebra in $R^\mrm{pre}$.
\end{theorem}

For an algebra $B$ in $\YD(G)$ the total center $Z_{tot}(B)$ of $B$ is the maximal subalgebra for which the two diagrams
\[
\xymatrix{
Z\ot B\ar[rr]^{\mfk{c}}\ar[dr]_{mult} & & B\ot Z\ar[dl]^{mult}\\
 & B &
} \ \ 
\xymatrix{
B\ot Z\ar[rr]^{\mfk{c}}\ar[dr]_{mult} & & Z\ot B\ar[dl]^{mult}\\
 & B &
}
\]
commute.

\begin{proposition}[{\cite[Theorem 3.3]{andruskiewitschschneider10}}]\label{prop:Z0centr}
Consider $Z_0$ in $R^\mrm{pre}$, and take $\mfk{c}=\mfk{c}_{R^\mrm{pre},R^\mrm{pre}}$.
\begin{enumerate}
\item[(i)] The restriction of the braiding $\mfk{c}$ to $Z_0\ot R^\mrm{pre}$ is an involution, i.e. $\mfk{c}|_{Z_0\ot R^\mrm{pre}}=(\mfk{c}|_{R^\mrm{pre}\ot Z_0})^{-1}$.
\item[(ii)] The subalgebra $Z_0$ is contained in the total center of $R^\mrm{pre}$, $Z_0\subset Z_{tot}(R^\mrm{pre})$.
\end{enumerate}
\end{proposition}

We note that in the case of the (classical) quantum De Concini-Kac-style Borel $U_q^{DK}(\mfk{b})$, the elements $E_\gamma^{N_\gamma}$ are actually central.  However, in general this will not be the case.  One can view the centrality in the classical de Concini-Kac setting as a consequence of the fact that $\mfk{c}|_{\mbb{C} E_\alpha^{N_\alpha}\ot U^{DK}_q(\mfk{b})}$ happens to be the trivial swap.

\subsection{Some technical lemmas}

\begin{lemma}\label{prop:881}
The adjoint action of $R^\mrm{pre}$ on itself factors through the quotient $R$.
\end{lemma}

\begin{proof}
It suffices to show that the adjoint action restricted to $Z_0\subset R^\mrm{pre}$ is trivial, since the kernel of the projection $R^\mrm{pre}\to R$ is generated by the augmentation ideal for $Z_0$.  For any (homogeneous) $X\in Z_0$ and $a\in R^\mrm{pre}$ we have
\[
\begin{array}{rll}
\mrm{ad}_\mrm{sk}(X)(a) & =\sum_i \chi_a(g_{i_2})X_{i_1}a S(X_{i_2}) \\
 & =\sum_i \chi_a(g_{i_2})\chi_{i_2}(\mrm{deg}(a)) X_{i_1}S(X_{i_2})a & (\text{Prop.~\ref{prop:Z0centr} (ii)})\\
 & =\sum_i \chi_a(g_{i_2})\chi_a(g_{i_2})^{-1}X_{i_1}S(X_{i_2})a & (\text{Prop.~\ref{prop:Z0centr} (i)})\\
 &=(\sum_i X_{i_1}S(X_{i_2}))a\\
 & = \epsilon(X)a,
\end{array}
\]
where in the above calculation $g_{i_2}$ is the $G$-degree of $X_{i_2}$ and $\chi_{i_2}$ is the $G^\vee$-degree.  Hence $\mrm{ad}_\mrm{sk}|_{Z_0}$ factors through the counit, and the restriction of the adjoint action to $Z_0$ is trivial, as desired.
\end{proof}

Let us order the basis of primitives $P_\mrm{ord}=\{x_i\}_i$ so that the matrix $[a_{ij}]$ is block diagonal with each block a Cartan matrix of type $A$, $D$, $E$, etc.  We take
\[
S_\mrm{ord}:=TV/(\mrm{ad}_\mrm{sk}(x_i)(x_j):\ i<j),\ \ s_\mrm{ord}:=TV/(\mrm{ad}_\mrm{sk}(x_i)(x_j),x_i^{N_i}:\ i<j).
\]
These are both algebras in $\YD(G)$.  We let $c_i$ denote the images of the $x_i$ in $S_\mrm{ord}$ and/or $s_\mrm{ord}$.

\begin{lemma}
The projections $TV\to S_\mrm{ord}$ and $TV\to s_\mrm{ord}$ factor to give projections $R^\mrm{pre}\to S_\mrm{ord}$ and $R\to s_\mrm{ord}$ respectively.
\end{lemma}

\begin{proof}
In $S_\mrm{ord}$ we have $\mrm{ad}_\mrm{sk}(c_j)(c_j^mc_i)=(1-q_{jj}^{m+a_{ji}})c_j^{m+1}c_i$ for $i<j$, which implies by induction
\[
\mrm{ad}_\mrm{sk}(c_j)^{1-{a_{ji}}}(c_i)=c_j^{1-a_{ji}}c_i\prod_{m=0}^{-a_{ji}}(1-q_{jj}^{m+a_{ji}})=0.
\]
When $R$ has no exceptional relations the above relation is sufficient to produce the proposed surjection $R^\mrm{pre}\to S_\mrm{ord}$.  In the case of exceptional relations, one checks directly from the presentations of~\cite[Eq. 4.6, 4.13, 4.22, 4.27, 4.34, 4.41, 4.49]{andruskiewitschschneider02} that the relations $\mrm{ad}_\mrm{sk}(c_i)(c_j)$, for $i<j$, imply all additional relations for $R^{\mrm{pre}}$ as well.  If we consider the projection $S_\mrm{ord}\to s_\mrm{ord}$, the addition of the relations $c_i^{N_i}$ to $S_\mrm{ord}$ imply the relations $c_\gamma^{N_{\gamma}}$.  So we also get the projection $R\to s_\mrm{ord}$.
\end{proof}

\subsection{Proof of Theorem~\ref{thm:cart_type}}
\label{sect:proofcartanQ}

\begin{proof}[Proof of Theorem~\ref{thm:cart_type}]
Take $S=S_\mrm{ord}$.  We have the adjoint action of $R^\mrm{pre}$ on itself, which induces an action of $R^\mrm{pre}$ on the braided symmetric algebra $S$.  Since the action of $R^\mrm{pre}$ on itself factors through $R$, the induced action on $S$ also factors to give a well-defined action of $R$ on $S$.  The generators $x_i$ in this case act as the adjoint operators $\mrm{ad}_\mrm{sk}(c_i)$.  We integrate the natural action of $G$ as well to get a well-defined action of $H=R\rtimes G$, which gives $S$ a well-defined $H$-module algebra structure (see Lemma~\ref{lem:YD_act}).
\par

We note that the restriction of the action $H\to \End_k(S)$ produces an embedding $V\to \End_k(S)$, where $V=R_1$ is the space of primitives in $R$.  To see this clearly, note that for any linear combination $v=\sum_i \kappa_i x_i$, and $i_{v}$ maximal in the ordered basis $P_\mrm{ord}$ such that $\kappa_{i_v}\neq 0$, we have
\[
v\cdot c_{i_v}=\kappa_{i_v}\mrm{ad}_\mrm{sk}(c_{i_v})(c_{i_v})=(1-q_{i_vi_v})\kappa_{i_v}c_{i_v}^2\neq 0.
\]
The action of $H$ will however not be inner faithful in general, as $G$ may not act faithfully on $S$.
\par

We have the additional action of $H$ on $\mbb{C}[w_\mu:\mu\in Y]$ given simply by the Hopf projection $H\to \mbb{C}[G]$ and the prescribed $G$-action on $\mbb{C}[w_\mu:\mu\in Y]$, $g\cdot w_\mu=\mu(g)w_\mu$.  We can therefore let $H$ act diagonally on the tensor product
\[
\mbb{C}[w_\mu:\mu\in Y]\ot S.
\]
Via the vector space equality 
\[
\mbb{C}[w_\mu:\mu\in Y]\ot S=\mbb{C}[w_\mu:\mu\in Y]\underline{\ot} S=A
\]
we get an $H$-action on $A$, which we claim gives it the structure of an $H$-module algebra.  To show this it suffices to show that the multiplication is $G$-linear and $R$-linear independently.
\par

The fact that the multiplication on $A$ is a map of $G$-representations follows from the fact that $A$ is an algebra object in $\YD(G)$.  For $R$-linearity it suffices to show that the braiding $\mfk{c}:S\ot \mbb{C}[w_\mu:\mu\in Y]\to\mbb{C}[w_\mu:\mu\in Y]\ot S$ is a map of $R$-modules, since $S$ and $\mbb{C}[w_\mu:\mu\in Y]$ are both $R$-module algebras independently.  However, this is clear as $\mbb{C}[w_\mu:\mu\in Y]$ is a trivial $R$-module.  Whence we find that $A$ is an $H$-module algebra, as proposed.  We then get an induced action of $H$ on the fraction field $Q=\Frac(A)$ by Theorem~\ref{thm:indQ}.
\par

The fact that the $H$-action on $Q$ is inner faithful when $Y$ generates $G^\vee$ follows by Corollary~\ref{cor:fate}, since the restrictions $G\to \End_k(A)$ and $V\to \End_k(A)$ are both injective.
\end{proof}

\section{Actions for (generalized) quantum groups}

We consider cocycle deformations of the Cartan type algebras considered in the previous section.  The primary example of such an algebra is the small quantum group $u_q(\mfk{g})$ associated to a simple Lie algebra and root of unity $q$.  However, more generally, one has the pointed Hopf algebras $u(\mcl{D})$ of Andruskiewitsch and Schneider.  These algebras are determined by a combinatorial data $\mcl{D}$ consisting of a collection of Dynkin diagrams and a so-called linking data for these diagrams.
\par

We produce actions of the Hopf algebras $u(\mcl{D})$ on central simple division algebras which are constructed from their Angiono-de Concini-Kac form $U(\mcl{D})$.  This action is inner faithful if and only if the skew primitives in $U(\mcl{D})$, considered as a representation of the grouplikes under the adjoint action, tensor generate $\operatorname{Rep}(G(u(\mcl{D})))$.  In the case of a classical quantum group $u_q(\mfk{g})$ we construct a faithful action on a central simple algebra via quantum function algebras, without imposing restrictions on the interactions of grouplikes and skew primitives.

\subsection{Actions for $u(\mcl{D})$}

Let $R=B(V)$ be of finite Cartan type.  Take $V$ in $\YD(G)$ for some abelian $G$ and consider the bosonization $H=R\rtimes G$.  Take a basis $\{x_1,\dots, x_\theta\}$ for $V$ consisting of $G\times G^\vee$-homogeneous elements.  Let $g_i$ be the $G$-degree of $x_i$.
\par

We can consider $V$ as object in $\YD(\mbb{Z}^\theta)$ and take
\[
H^\mrm{pre}:=R^\mrm{pre}\rtimes \mbb{Z}^\theta.
\]
Specifically, $\mbb{Z}^\theta$ has generators $t_i$, we have the group map $\mbb{Z}^\theta\to G$, $t_i\mapsto g_i$, and we let $\mbb{Z}^\theta$ act on $V$ via this group map.  We take each $x_i\in V$ to be homogeneous of $\mbb{Z}^\theta$-degree $t_i$.

\begin{lemma}\label{lem:513}
For $R=B(V)$, and $V$ of Cartan type as above, the algebra $H^\mrm{pre}$ is a domain which is finite over its center.
\end{lemma}

\begin{proof}
Recall that $R^\mrm{pre}$ is finite over the subalgebra $Z_0$, which is generated by the $x_\gamma^{N_\gamma}$ and lies in the total braided center by Proposition~\ref{prop:Z0centr}.  Hence $R^\mrm{pre}$ is finite over the central subalgebra $Z'_0$ generated by the powers $x_\alpha^{\mrm{exp}(G)}$.  If we take $K$ to be the kernel of the projection $K\to \mbb{Z}^\theta\to G$, it follows that $H^\mrm{pre}$ is finite over $Z_0'\ot \mbb{C}[K]$.
\par

We show that $H^\mrm{pre}$ is a domain.  We first show that $R^\mrm{pre}$ is a domain.  Just as in~\cite[\S 1.7, Proposition 1.7]{deconcinikac90} (cf.~\cite[Lemma 2.4]{mpsw10}), one can filter $R^\mrm{pre}$ via a normal ordering on the positive roots for the root system associated to $V$ to get that $\mrm{gr}R^\mrm{pre}$ is a skew polynomial ring generated by the $x_\alpha$.  In particular, $\mrm{gr}R^\mrm{pre}$ is a domain, and hence $R^\mrm{pre}$ is a domain.  By considering the $\mbb{Z}^\theta$-grading on $H^\mrm{pre}$ given directly by the $\mbb{Z}^\theta$ factor, we see that $H^\mrm{pre}$ is a domain as well.
\end{proof}

We note that any Hopf $2$-cocycle $\sigma:H\ot H\to \mbb{C}$ restricts to a Hopf $2$-cocycle on $H^\mrm{pre}$, via the projection $H^\mrm{pre}\to H$.  Hence we can consider for any such $\sigma$ the twist $H^\mrm{pre}_\sigma$ and Hopf projection $H^\mrm{pre}_\sigma\to H_\sigma$.

\begin{lemma}\label{lem:526}
Consider any $2$-cocycle $\sigma:H\ot H\to \mbb{C}$ with trivial restriction $\sigma|_{G\times G}=1$.  Then the following holds:
\begin{enumerate}
\item[(i)] The cocycle deformation $H^\mrm{pre}_\sigma$ is (still) a domain.
\item[(ii)] $H^\mrm{pre}_\sigma$ is finite over its center.
\item[(iii)] The adjoint action of $H^\mrm{pre}_\sigma$ on itself factors through $H_\sigma$.
\end{enumerate}
\end{lemma}

\begin{proof}
(i) By considering the associated graded algebra $\mrm{gr} H^\mrm{pre}_\sigma$ with respect to the coradical filtration, and Lemma~\ref{lem:513}, we see that $H^\mrm{pre}_\sigma$ is a domain.  In particular, $\mrm{gr} H^\mrm{pre}_\sigma=H^\mrm{pre}$, which is a domain by Lemma~\ref{lem:513}.
\par

(ii) Let $\Pi$ be the kernel of the projection $\mbb{Z}^\theta\to G$, and take $\msc{Z}=Z_0\rtimes \Pi$.  Then we have an exact sequence of Hopf algebras $\msc{Z}\to H^\mrm{pre}\to H$.  Therefore
\[
\sigma|_{\msc{Z}\ot H^\mrm{pre}}=\sigma|_{H^\mrm{pre}\ot \msc{Z}}=\epsilon
\]
and $H^\mrm{pre}_\sigma=H^\mrm{pre}$ as a $\msc{Z}$-bimodule.  In particular $H^\mrm{pre}_\sigma$ is a finite module over $\msc{Z}$.  Since $\msc{Z}$ is finite over the central subalgebra generated by the kernel $\Pi$ of the projection $\mbb{Z}^\theta\to G$ and the $\exp(G)$-th powers of the generators for $R^\mrm{pre}$, we see that $H^\mrm{pre}_\sigma$ is finite over its center.
\par

(iii) We note that the subalgebra $\msc{Z}=Z_0\rtimes \Pi$ in $H^\mrm{pre}_\sigma$ is a Hopf subalgebra.  Since $H^\mrm{pre}_\sigma=H^\mrm{pre}$ as a $\msc{Z}$-bimodule, it follows that the adjoint action of $\msc{Z}$ on $H^\mrm{pre}_\sigma$ is still trivial, by Proposition~\ref{prop:881}.  Whence the adjoint action of $H^\mrm{pre}_\sigma$ on $H^\mrm{pre}_\sigma$ restricts trivially to $\msc{Z}$, and from the exact sequence $\msc{Z}\to H^\mrm{pre}_\sigma\to H_\sigma$ we see that the adjoint action factors through $H_\sigma$.
\end{proof}

\begin{theorem}\label{thm:cocyc_tw}
Suppose that $V\in \YD(G)$ is of finite Cartan type, and that $V$ (tensor) generates $\operatorname{Rep}(G)$.  Then for any $2$-cocycle $\sigma$ of $H=B(V)\rtimes G$ with $\sigma|_{G\times G}=1$, the adjoint action of $H_\sigma$ on $H^\mrm{pre}_\sigma$ is inner faithful.  Consequently, the induced action of $H_\sigma$ on the central simple division algebra $\Frac(H^\mrm{pre}_\sigma)$ is inner faithful.
\end{theorem}

\begin{proof}
The fact that $V$ generates $\operatorname{Rep}(G)$ implies that all characters for $G$ appear in the decomposition of $H^\mrm{pre}$ into simples, under the adjoint action.  So $G$ acts faithfully on $H^\mrm{pre}$.  Triviality of the restriction $\sigma|_{G\times G}$ implies that the grading $\mrm{gr}H_\sigma^\mrm{pre}$ with respect to the coradical filtration is the bosonization $H^\mrm{pre}$.  Semisimplicity of $\mbb{C}[G]$ then implies an isomorphism of $G$-representations $H_\sigma^\mrm{pre}\cong H^\mrm{pre}$.  So we see that $G$ acts faithfully on $H_\sigma^\mrm{pre}$.
\par

All that is left is to verify that the restriction of the adjoint action $H_\sigma\to \End_\mbb{C}(H_\sigma^\mrm{pre})$ to the space of nontrivial $(g,1)$-skew primitives $\Prim_g(H_\sigma^\mrm{pre})'$ is injective.  Note that $H_\sigma^\mrm{pre}$ is a $G$-graded {\it vector space} (not algebra) with grading induced by comultiplication and projection $H_\sigma^\mrm{pre}\to H_\sigma^\mrm{pre}\ot H_\sigma^\mrm{pre}\to H_\sigma^\mrm{pre}\ot\mbb{C}[G]$.  Choose any such primitive $v$ and $a$ of trivial $G$-degree, i.e.\ $a\in B(V)\subset H_\sigma^\mrm{pre}$.  Note that $v\in V$, and hence $v$ has a canonical lift to $H_\sigma$.  We have
\[
v\cdot_{ad}a=\sigma(v,a_1)a_2+\sigma(g,a_1)va_2+\sigma(g,a_1)\tilde{g}a_2\sigma^{-1}(v,a_3)+\text{elements in degree }G-\{e\}.
\]
So we see that it suffices to show that the $e$-degree term is nonvanishing.
\par

Take $i$ minimal with $a\in F_iH_\sigma$, where we filter with respect to the coradical filtration.  Then, since $\mrm{gr} H_\sigma=H$,
\[
\sigma(v,a_1)a_2+\sigma(g,a_1)va_2+\sigma(g,a_1)a_2\sigma^{-1}(v,a_3)=va\mod F_i H_\sigma.
\]
Since $H$ is a domain, $va$ is nonzero, and we conclude $v\cdot_{ad}a$ is nonzero.  It follows that the restriction of the adjoint action to each $\Prim_g(H_\sigma)'$ is injective, and the adjoint action of $H_\sigma$ on $H_\sigma^\mrm{pre}$ is inner faithful by Lemma~\ref{lem:fate}.
\end{proof}

We are particularly interested in the generalized quantum groups $u(\mcl{D})=u(\mcl{D},\lambda,\mu)$ of Andruskiewitsch and Schneider~\cite{andruskiewitschschneider10}.  These algebras are determined by a collection of Dynkin diagrams and a ``linking data" $\mcl{D}=(\mcl{D},\lambda,\mu)$ between the Dynkin diagrams.  As far as the above presentation is concerned, we have
\[
u(\mcl{D})=(B(V)\rtimes G)_\sigma=H_\sigma
\]
for a finite Cartan type $V$ and a cocycle $\sigma$ which restricts trivially to the grouplikes~\cite[Section 5.2, 5.3]{andruskiewitschschneider10},~\cite[Corollary 1.2]{angionoiglesias16}.  A direct application of Theorem~\ref{thm:cocyc_tw} yields

\begin{corollary}\label{cor:uD}
Suppose $V\in \YD(G)$ is of finite Cartan type, and that $V$ generates $\mrm{Rep}(G)$.  Then the generalized quantum group $u(\mcl{D})$ associated to any linking data $\mcl{D}$ admits an inner faithful action on a central simple division algebra.
\end{corollary}

\begin{remark}
The supposition that $V$ generates $\mrm{Rep}(G)$ is a serious restriction.  For classical quantum groups $u_q(\mfk{g})$, for example, the space of skew primitives generates $\mrm{Rep}(G)$ if and only if $q$ is relatively prime to the determinant of the Cartan matrix for $\mfk{g}$.  For generalized Taft algebras $T(n,m,\alpha)$, we have such generation if and only if $m=n$.
\end{remark}

\subsection{More refined actions for standard quantum groups}

Let $q$ be an odd root of $1$, $\mfk{g}$ be a simple Lie algebra, and $u_q(\mfk{g})$ be the corresponding small quantum group.  We assume additionally that the order of $q$ is coprime to $3$ when $\mfk{g}$ is of type $G_2$.

\begin{proposition}\label{prop:uqg}
There is an inner faithful action of $u_q(\mfk{g})$ on $\Frac(\O_q(\mbb{G}))$, where $\mbb{G}$ is the simply-connected, semisimple, algebraic group with Lie algebra $\mfk{g}$.  Furthermore, this action is Hopf-Galois.  In particular, $u_q(\mfk{g})$ acts inner faithfully on a central simple division algebra.
\end{proposition}

\begin{proof}
By definition, $\O_q(\mbb{G})$ is the finite dual of the Lusztig, divided powers, quantum group $U_q(\mfk{g})$.  We have the action of $u_q(\mfk{g})$ on $\O_q(\mbb{G})$ by left translation
\[
x\cdot f:=(a\mapsto f(ax))\ \ \text{for }x\in u_q(\mfk{g}),\ f\in \O_q(\mbb{G}).
\]
This action is faithful as it reduces to a faithful action of $u_q(\mfk{g})$ on the quotient $u_q(\mfk{g})^\ast$.
\par

The exact sequence $\mbb{C}\to u_q(\mfk{g})\to U_q(\mfk{g})\to U(\mfk{g})\to \mbb{C}$~\cite{lusztig90II} gives an exact sequence
\[
\mbb{C}\to \O(\mbb{G})\to \O_q(\mbb{G})\to u_q(\mfk{g})^\ast\to \mbb{C}.
\]
(By an exact sequence $\mbb{C}\to A\to B\to C\to \mbb{C}$ we mean that $A\to B$ is a faithfully flat extension with $B\ot_A\mbb{C}\cong C$, and that $A$ is the $C$-coinvariants in $B$.)  The subalgebra $\O(\mbb{G})$ is central in $\O_q(\mbb{G})$, and $\O_q(\mbb{G})$ is finite over $\O(\mbb{G})$.  Furthermore, $\O_q(\mbb{G})$ is a domain~\cite[III.7.4]{browngoodearl00}.  So we take the algebra of fractions $\Frac(\O_q(\mbb{G}))$ to arrive at a central simple division algebra on which $u_q(\mfk{g})$ acts inner faithfully.
\par

As for the Hopf-Galois property, faithful flatness of $\O_q(\mbb{G})$ over $\O(\mbb{G})$ implies that $\O_q(\mbb{G})$ is a locally free $\O(\mbb{G})$-module, and also $\O(\mbb{G})=\O_q(\mbb{G})^{u_q(\mfk{g})}$~\cite[Theorem 2.1]{masuokawigner94}.  From the equality $\Frac(\O_q(\mbb{G}))=\Frac(\O(\mbb{G}))\ot_{\O(\mbb{G})}\O_q(\mbb{G})$ one calculates
\[
\mrm{rank}_{\Frac(\O(\mbb{G}))}\operatorname{Frac}(\O_q(\mbb{G}))=\mrm{rank}_{\O(\mbb{G})}\O_q(\mbb{G})=\dim(u_q(\mfk{g}))
\]
and $\Frac(\O(\mbb{G}))=\Frac(\O_q(\mbb{G}))^{u_q(\mfk{g})}$.  It follows that the given extension is Hopf-Galois by Theorem~\ref{thm:Gal}.
\end{proof}

\section{Proof of Theorem~\ref{thm:TQstructure}}
\label{sect:proof}

We first establish some general information regarding skew derivations of central simple algebras, then provide the proof of Theorem~\ref{thm:TQstructure}.

\subsection{Bimodules in Yetter-Drinfeld categories and skew derivations}
\label{sect:YD}

Given a field $K$ we write $\YD_K(G)$ for the category of Yetter-Drinfeld modules over the group algebra $KG$.  We always assume $K$ is of characteristic $0$.

\begin{lemma}\label{lem:161}
Let $A$ be an algebra in $\YD_K(G)$.  There is an equivalence of categories between the subcategory of $A$-bimodules in $\YD_K(G)$ and right $A^{\underline{op}}\underline{\ot}_K A$-modules in $\YD_K(G)$.  This equivalence takes a bimodule $M$ to the Yetter-Drinfeld module $M$ along with the right $A^{\underline{op}}\underline{\ot}_KA$-action $m\cdot (a\ot b):=(m_{-1}a)m_0b$.
\end{lemma}

\begin{proof}
Straightforward direct check.
\end{proof}

Recall that in characteristic $0$, a finite-dimensional semisimple $K$-algebra $A$ is separable over $K$.

\begin{lemma}\label{lem:698}
Let $G$ be an abelian group and $A$ be an algebra in $\YD(G)$, which is semisimple as a $\mbb{C}$-algebra.  Let $K$ be a central invariant subfield in $A$ over which $A$ is finite.  Then the algebra $A$ is projective as an $A^{\underline{op}}\underline{\ot}_K A$-module.
\end{lemma}

\begin{proof}
Since $G$ is abelian, the Yetter-Drinfeld structure on $A$ is equivalent to a $G\times G^\vee$-grading on $A$.  Take $G'=G\times G^\vee$.  We claim that $A\ot_K A\to A$ admits a homogeneous degree $0$ section, as a map of bimodules.  To see this one simply takes an arbitrary separability idempotent $e$ and expands $e=\sum_{g,h\in G'}e_g\ot e_h$ with each $e_g\ot e_h\in A_g\ot_K A_h$.  Take $e'=\sum_g e_g\ot e_{g^{-1}}$.  Since the multiplication on $A$ is  homogeneous we see that $m(e')=1$.  Furthermore, since the multiplication on the right and left of $A\ot A$ preserves the grading, we see that $a e'=e'a$ for each homogeneous $a\in A$, and hence each $a\in A$.  So the map $A\to A\ot_K A$, $1\mapsto e'$, provides a degree $0$ splitting of the multiplication map.  By Lemma~\ref{lem:161} we see that the projection
\[
A^{\underline{op}}\underline{\ot}_K A\to A,\ \ a\ot b\mapsto ab
\]
is split as well, and hence that $A$ is projective over $A^{\underline{op}}\underline{\ot}_K A$.
\end{proof}

\begin{lemma}\label{lem:inn_der}
Take $G$ abelian, and let $A$ be a $G$-module central semisimple algebra.  Let $K$ be a central invariant subfield over which $A$ is finite, and let $M$ be a $K$-central $A$-bimodule in $\mrm{Rep}(G)$.  Then every $K$-linear, homogeneous, $(g,1)$-skew derivation $f:A\to M$, for $g\in G$, is inner.
\end{lemma}

By homogeneous we mean the following: if we decompose $A$ and $M$ into character spaces $A=\oplus_\mu A_\mu$, $M=\oplus_\mu M_\mu$, then $f(A_\mu)\subset M_{\mu\sigma}$ for some fixed $\sigma\in G^\vee$.  So $f$ is homogeneous of degree $\sigma$ here.  By an inner skew derivation we mean there is $c\in M_\sigma$ so that $f=[c,-]_\mrm{sk}:a\mapsto (ca-(g\cdot a)c)$.

\begin{proof}
Take $\sigma=\deg_{G^\vee}(f)$.  We choose a non-degenerate form $b:G\times G\to \mbb{C}^\times$ and let $G^\vee$ act on $A$ and $M$ via the isomorphism $f_b:G^\vee\to G$ provided by the form.  Then we decompose $A$ and $M$ into character spaces $A=\oplus_\mu A_\mu$ and $M=\oplus_\mu M_\mu$, and the corresponding $G$-gradings $A=\oplus_g A_g$ and $M=\oplus_g M_g$ are such that $A_g=A_{\mu}$ and $M_g=M_{\mu}$ for $\mu$ with $g=f_b(\mu)$.  There is a unique shift $M[h]$ of the $G$-grading on $M$ so that $M_{\sigma}=(M[h])_g$.  In this way $A$ and $M[h]$ are objects in $\YD_K(G)$, and $M[h]$ is an $A$-bimodule in $\YD_K(G)$.
\par

Consider $M[h]$ as an $A^{\underline{op}}\underline{\ot}_KA$-module.  As in~\cite[Proposition 3.3(1)]{negron}, one can show that
\[
\mathrm{Ext}^1_{A^{\underline{op}}\underline{\ot}_K A}(A,M[h])=\{\mathrm{Skew\ derivations}\}/\{\mathrm{Inner\ derivations}\}.
\]
Since $A$ is separable, this cohomology group vanishes.  Whence we conclude that each skew derivation of $M$ is inner.
\end{proof}

\subsection{Proof of Theorem~\ref{thm:TQstructure}}

We consider again the algebra $T(n,m,\alpha)$.  We will need the following result.

\begin{proposition}[{\cite[Proposition 3.9]{etingof15}}]\label{prop:735}
Suppose $H$ is a finite-dimensional Hopf algebra acting on an algebra $A$ which is finite over its center.  Then $A$ is finite over the invariant part of its center $Z(A)^H=Z(A)\cap A^H$.
\end{proposition}

From a $G$-module algebra $A$, an element $c\in A_i$, and fixed $g\in G$, we let $[c,-]_\mrm{sk}:A\to A$ denote the endomorphism $[c,a]_\mrm{sk}:=ca-(g\cdot a)c$.  We now prove Theorem~\ref{thm:TQstructure}.

\begin{proof}[Proof of Theorem~\ref{thm:TQstructure}]
Take $G=G(T(n,m,\alpha))=\langle g\rangle$, and $\zeta$ a primitive $n$-th root of $1$ with $\zeta^{n/m}=q$.  We fix $A$ a $G$-module central simple algebra, which we decompose as $A=\oplus_{i=1}^nA_i$ so that $g|_{A_i}=\zeta^i\cdot-$.  We claim that, for an arbitrary element $c\in A_{n/m}$, we have
\begin{equation}\label{eq:743}
[c,-]^m_\mrm{sk}(a)=c^ma-\zeta^{m|a|}ac^m.
\end{equation}
The skew commutator here employs the action of the generator $g$.  The equality~\eqref{eq:743} will imply the desired result, as for any $T(n,m,\alpha)$-action on $A$, which extends the given action of $G$, we will have $x\cdot-=[c,-]_\mrm{sk}$ for some $c\in A_{n/m}$ by Lemma~\ref{lem:inn_der}.  In our application of Lemma~\ref{lem:inn_der} here we take $K=Z(A)^T$.  So we seek to prove~\eqref{eq:743}.
\par

We note that
\[
q^{m(m-1)/2}=\left\{\begin{array}{ll} q^{m/2}=-1 & \text{when $m$ is even}\\ 1 & \text{when $m$ is odd}\end{array}\right.
=(-1)^{m+1}.
\]
So the desired relation~\eqref{eq:743} can be rewritten as
\begin{equation}\label{eq:comm_rel2}
[c,-]^m_\mrm{sk}(a)=c^ma+(-1)^m\zeta^{m|a|}q^{m(m-1)/2}ac^m.
\end{equation}

We have directly
\begin{equation}\label{eq:242}
[c,-]^m_\mrm{sk}(a)=c^ma+\sum_{l=1}^m (-1)^l\zeta^{l|a|}\omega_lc^{m-l}a c^l,
\end{equation}
for coefficients $\omega_i\in \mbb{Q}(\zeta)$.  The coefficient $\omega_l$ can be deduced as follows: Each $c$ appearing on the right of $c^{m-l}ac^l$ indicates an integer $i$ so that at the $i$-th iteration of 
\[
\begin{array}{rl}
[c,[c,-]^{i-1}_\mrm{sk}(a)]_\mrm{sk}&=c[c,-]^{i-1}_\mrm{sk}(a)-\zeta^{|a|+(i-1)m}[c,-]^{i-1}_\mrm{sk}(a)c\\
&=c[c,-]^{i-1}_\mrm{sk}(a)-\zeta^{|a|}q^{(i-1)}[c,-]^{i-1}_\mrm{sk}(a)c
\end{array}
\]
we take the summand $q^{(i-1)}[c,-]^{i-1}_\mrm{sk}(a)c$.  Each choice of $l$ such distinct positions $\{k_1,\dots,k_l\}\subset \{1,\dots, m\}$ contributes a summand with $q$-coefficient $(\prod_{j=1}^lq^{(k_j-1)})$.  Take $[m-1]=\{0,\dots, m-1\}$.  Considering all possible choices for the subset $\{k_1,\dots, k_l\}$ gives
\begin{equation}\label{eq:771}
\omega_l=\sum_{1\leq k_1< \dots < k_l\leq m}\left(\prod_{j=1}^lq^{(k_j-1)}\right)=\sum_{\underset{|I|=l}{I\subset[m-1]}}(\prod_{i\in I}q^{i}).
\end{equation}
When $l=1$ the above sum gives 
\[
\omega_1=(1+q+\dots +q^{m-1})=\frac{1-q^m}{1-q}=0.
\]
and
\[
\omega_m=\prod_{i=0}^{m-1}q^i=q^{\sum_{i=1}^m i}=q^{m(m-1)/2}
\]
We want to show $\omega_l=0$ for all $0<l<m$.

We can rewrite the sum of products~\eqref{eq:771} as a product of sums
\begin{equation}\label{eq:785}
\omega_l=\frac{1}{l!}(\sum_{j_1\in [m-1]}q^{j_1}(\sum_{j_2\in [m-1]-\{j_1\}}q^{j_2}(\dots (\sum_{j_l\in [m-1]-\{j_1,\dots,j_l\}}q^{j_l})\dots))).
\end{equation}
Take
\[
\omega_l(j)=\sum_{j\notin I}(\prod_{i\in I}q^{i})\ \ \text{and}\ \ \omega_l'(j)=\sum_{j\in I}(\prod_{i\in I}q^{i}),
\]
where in the first sum $I$ runs over size $l$ subsets of $[m-1]$ which do not contain the given $j\in[m-1]$, and the second sum runs over subsets containing $j$.  Then
\begin{equation}\label{eq:275}
\omega_l=\omega_l(j)+\omega_l'(j).
\end{equation}
Note that $\omega_l'(j)=q^j\omega_{l-1}(j)$, where $\omega_0(j)$ is formally taken to be $1$.  Then the expression~\eqref{eq:785} gives
\begin{equation}\label{eq:279}
\omega_l=\frac{1}{l!}\left(\sum_{j=0}^{m-1}q^{j}(l-1)!\!\ \omega_{l-1}(j)\right) =\frac{1}{l}\sum_{j=0}^{m-1}q^{j}\omega_{l-1}(j).
\end{equation}

We have already seen that $\omega_1=0$.  We take $l<m$ and suppose that $\omega_k=0$ for all $k<l$.  Then the decomposition $\omega_k=\omega_k(j)+\omega_k'(j)$ for all $j\in [m-1]$ implies
\[
\omega_k(j)=-\omega_k'(j)=-q^j\omega_{k-1}(j-1)
\]
for all $k<l$ and $j$.  Hence, from~\eqref{eq:279},
\[
\begin{array}{rl}
\omega_l & =l^{-1}\sum_{j=0}^{m-1}q^{j}\omega_{l-1}(j)\\
& =-l^{-1}\sum_{j=0}^{m-1}q^j\omega_{l-1}'(j)\\
& =-l^{-1}\sum_{j=0}^{m-1}q^{2j}\omega_{l-2}(j)\\
& =(-1)^2l^{-1}\sum_{j=0}^{m-1}q^{2j}\omega_{l-2}'(j)\\
&\vdots\\
&=(-1)^{l-1} l^{-1}\sum_{j=0}^{m-1}q^{lj}\omega_0(j)=(-1)^{l-1}l^{-1}(1-q^{lm})/(1-q^l)=0.
\end{array}
\]
Hence $\omega_l=0$ for all $l<m$.  One recalls our initial expression~\eqref{eq:242} to arrive finally at the desired equality $[c,-]^m_\mrm{sk}(a)=c^ma-\zeta^{m|a|}ac^m$.
\end{proof}

\section{Coradically graded algebras and universal actions}
\label{sect:univ}

Let us fix now a coradically graded, pointed Hopf algebra $H$ with abelian group of grouplikes.  We may write $H=B(V)\rtimes G$, with $G$ abelian and $V$ in $\YD(G)$.  Fix also a homogeneous basis $\{x_i\}_i$ for $V$ with respect to the $G\times G^\vee$-grading provided by the Yetter-Drinfeld structure.

\subsection{The universal algebra}

We consider the (Hopf) free algebra $TV$ in $\YD(G)$ as a module algebra over itself under the adjoint action
\[
a\cdot_\mrm{adj}b:=a_1\big((a_2)_{-1}b)S\big((a_2)_0\big).
\]
Consider a presentation $B(V)=TV/(r_1,\dots,r_l)$ with each $r_i$ homogeneous with respect to the $G\times G^\vee$-grading, as well as the grading on $TV$ by degree.
\par

Define $A_\mrm{univ}$ as the quotient
\[
A_\mrm{univ}=A_\mrm{univ}(V):=TV/(r_i\cdot_\mrm{adj}a:1\leq i\leq l, a\in TV).
\]
We note that $A_\mrm{univ}$ is a connected graded algebra in $\YD(G)$, as all relations can be taken to be homogeneous with respect to all gradings.  Furthermore, the adjoint action of the free algebra on itself induces an action of $TV$ on $A_\mrm{univ}$.  We let $c_i$ denote the image of $x_i\in V$ in $A_\mrm{univ}$.

\begin{lemma}
The adjoint action of $TV$ on $A_\mrm{univ}$ induces an action of $B(V)$ on $A_\mrm{univ}$.  This action is specified on the generators by $x_i\cdot a=[c_i,a]_\mrm{sk}:=c_ia-(g_i\cdot a)c_i$.
\end{lemma}

\begin{proof}
Evident by construction.
\end{proof}

Since each relation for $B(V)$ in $TV$ must act trivially on $A_\mrm{univ}$ we have immediately

\begin{corollary}
For any $r$ in the kernel of the projection $TV\to B(V)$, and arbitrary $a\in TV$, $A_\mrm{univ}$ has the relation $r\cdot_{\rm adj}a=0$.  In particular, the $B(V)$-module algebra $A_\mrm{univ}$ is independent of the choice of relations for $B(V)$.
\end{corollary}

\begin{definition}
For given $V$ in $\YD(G)$, with $G$ abelian, we call $A_\mrm{univ}(V)$ the {\it universal algebra} for $V$.
\end{definition}

We would like to construct from $A_\mrm{univ}$ central simple $H$-division algebras, and therefore would like to develop means of understanding when $A_\mrm{univ}$ itself is finite over its center.

\begin{lemma}\label{lem:956}
Suppose the kernel $I$ of the projection $TV\to B(V)$ contains a right coideal subalgebra $\msc{R}\subset I$ such that
\begin{enumerate}
\item[(a)] $\msc{R}$ is a graded subalgebra in $\YD(G)$,
\item[(b)] $\msc{R}$ is finitely generated and
\item[(b)] the quotient $TV/(\msc{R}^+)$ is finite-dimensional.
\end{enumerate}
Then the algebra $A_\mrm{univ}(V)$ is finitely presented and finite over its center.
\end{lemma}

\begin{proof}
Enumerate a homogeneous generating set $\{r_1,\dots, r_d\}$ for $\msc{R}$.  By homogeneous we mean homogeneous with respect to the $G\times G^\vee$-grading as well as the $\mbb{Z}$-grading.  Define $B=TV/(\msc{R}^+)=TV/(r_1,\dots, r_d)$ and $A=TV/(r_i\cdot_\mrm{adj}a)_i$, where $a$ runs over homogeneous elements in $TV$.  Note that $B$ is a finite-dimensional Hopf algebra in $\YD(G)$, by hypothesis, and surjects onto $B(V)$.  Note also that $A$ surjects onto $A_\mrm{univ}$.
\par

Take $I_k$ to be the ideal in $TV$ generated by the relations $r_i\cdot_\mrm{adj}a$ for $r_i$ with $\deg(r_i)\leq k$, and homogeneous $a\in TV$.  Let $J_k$ be the ideal generated by the $[r_i,a]_\mrm{sk}=r_ia-(g a)r_i$ for $r_i$ with $\deg(r_i)\leq k$ and $a$ homogeneous, where $g=\deg_G(r_i)$.  Since each $[r_i,-]_{\mrm{sk}}$ is a skew derivation, $J_k$ is alternatively generated by the relations $[r_i,x_j]_{\mrm{sk}}$ for varying $i$ and $j$.  We would like to show $I_k=J_k$ for all $k$.  We have $I_1=J_1=0$.
\par

We have for each relation
\[
\Delta(r_i)=r_i\ot 1+1\ot r_i+\sum_{m} f_m\ot h_m,
\]
where the $f_m\in \msc{R}$ and the $h_m\in TV$, and $\deg(f_m),\deg(h_m)<\mrm{deg}(r_i)$, since $\msc{R}$ is coideal subalgebra.  Suppose we have $I_{k-1}=J_{k-1}$ for some $k$.  Then
\[
I_k=(r_i\cdot_{\mrm{adj}}a:\deg(r_i)=k)_{a\in TV}+I_{k-1}=(r_i\cdot_{\mrm{adj}}a:\deg(r_i)=k)_{a\in TV}+J_{k-1},
\]
and one also computes for $r_i$ of degree $k$,
\[
\begin{array}{rl}
r_i\cdot_\mrm{adj}a&=[r_i,a]_\mrm{sk}+\sum_m \chi_a(\mrm{deg}_G(h_m))f_maS(h_m)\\
&=r_ia+\chi_a(g)aS(r_i)+\sum_m \chi_a(g)af_mS(h_m)\ \ \mrm{mod}\ J_{\mrm{deg}(r_i)-1}\\
&=r_ia+\chi_a(g)a\big((r_i)_1S((r_i)_2)-r_i)\\
&=r_ia-\chi_a(g)ar_i\\
&=[r_i,a]_\mrm{sk},
\end{array}
\]
where in the above computation $\deg_G(r_i)=g$ and $\deg_{G^\vee}(a)=\chi_a$.  Hence $I_k=J_k$ and, by induction, we have 
\[
(r_i\cdot_\mrm{adj}a)_{i,a}=\cup_{k>0} I_k=\cup_{k>0} J_k=([r_i,x_j]_\mrm{sk})_{i,j}.
\]
The above identification provides a presentation
\begin{equation}\label{eq:1024}
A=TV/([r_i,a]_\mrm{sk})_{i,a}=TV/([r_i,x_j]_\mrm{sk})_{i,j}.
\end{equation}
\par

Let $\msc{R}'$ be the image of $\msc{R}$ in $A$.  Via the relations~\eqref{eq:1024} we see that $\msc{R}'$ is the quotient of a skew polynomial ring which is finite over its center, and also that $\msc{R}'$ is normal in $A$, in the sense that $(\msc{R}')^+A=A(\msc{R}')^+$.  Note that a bounded below $\mbb{Z}$-graded module $M$ over a $\mbb{Z}_{\geq 0}$-graded algebra $T$ with $T_0=\mbb{C}$ is finitely generated if and only if the reduction $\mbb{C}\ot_T M$ is finite-dimensional.  So we see that $A$ is finite over $\msc{R}'$, and hence finite over its center, as the reduction $\mbb{C}\ot_{\msc{R}'}A=B$ is finite-dimensional by hypothesis.
\par

The center of $\msc{R}'$ is finite over $\mbb{C}[r^{\exp(G)}_i:1\leq i\leq d]$ and hence finitely generated.  In particular, the center of $\msc{R}'$ is Noetherian.  As $A$ is finite over $Z(\msc{R}')$ it follows that any ideal in $A$ is finitely generated as well.  Whence the kernel of the surjection $A\to A_\mrm{univ}$ is finitely generated, and we see that $A_\mrm{univ}$ is finitely presented.
\end{proof}

\begin{remark}
In the notation of Lemma~\ref{lem:956}, one can produce coideal subalgebras in $I\subset TV$ by considering, for example, subalgebras generated by coideals in $TV$ which are contained in $I$.
\end{remark}

The most immediate way for the hypotheses of Lemma~\ref{lem:956} to be satisfied is if a generating set of relations for $B(V)$ can, in its entirety, be chosen to generate a coideal subalgebra in $TV$.

\begin{lemma}\label{lem:AunivPI}
Suppose there is a choice of homogeneous relations $\{r_1,\dots, r_d\}$ for $B(V)$ so that the subalgebra $\msc{R}$ generated by the $r_i$ in $TV$ forms a coideal subalgebra.  (For example, this occurs when the relations for $B(V)$ can be chosen to be primitive.)  Then $A_\mrm{univ}$ is finite over its center, and has a presentation $A_\mrm{univ}=TV/([r_i,x_j]_\mrm{sk})_{i,j}$.
\end{lemma}

\begin{proof}
The fact that $A_\mrm{univ}$ is finite over its center follows by Lemma~\ref{lem:956}.  The presentation by skew commutators was already provided in the proof of Lemma~\ref{lem:956}.
\end{proof}

In non-Cartan, diagonal, type the stronger hypotheses of Lemma~\ref{lem:AunivPI} are not always met.  (There are certainly examples in which they are met, however.  See Section~\ref{sect:ane}.)  Indeed, one can show for some simple super-type algebras that $A_\mrm{univ}$ does not have the desired commutator relations.  In some more regular settings, however, we expect that the conditions of Lemma~\ref{lem:AunivPI} will be met.  One can prove, for example, that this occurs for the quantum Borel in small quantum $\mfk{sl}_3$ at $q$ a $3$-rd root of $1$.

\subsection{Central simple division algebras via the universal algebra}

Take $A_\mrm{univ}=A_\mrm{univ}(V)$, as above, and $H=B(V)\rtimes G$.  Consider any field $K$ with a $G$-action, which we consider as an algebra in $\YD(G)$ by taking the trivial $G$-grading, and also as a trivial $B(V)$-module algebra.  We may take the tensor product $K\underline{\ot}A_\mrm{univ}$ to get a well-defined $B(V)$-module algebra in $\YD(G)$ (cf.\ proof of Theorem~\ref{thm:cart_type}).  Consider now any quotient
\[
A(K,I):=K\underline{\ot}A_\mrm{univ}/I
\]
via a prime $G$-ideal $I$ such that $A(K,I)$ is (a domain which is) finite over its center.  Since $B(V)$ acts by skew commutators on $K\underline{\ot}A_\mrm{univ}$, any such ideal will additionally be an $H=B(V)\rtimes G$-ideal.  In this case the ring of fractions
\[
Q(K,I):=\Frac(K\underline{\ot}A_\mrm{univ}/I)
\]
is a central simple division algebra on which $B(V)$ acts faithfully, by~\cite[Theorem 2.2]{skryabinoystaeyen06}.

\begin{definition}
A pair $(K,I)$ of an field $K$ with a $G$-action and a prime $G$-ideal $I$ in $K\underline{\ot}A_\mrm{univ}$ is called a pre-faithful pair if the quotient $A(K,I)$ is finite over its center.  A pre-faithful pair is called faithful if the $H$-action on $A(K,I)$ is inner faithful.
\end{definition}

Note that when $A_\mrm{univ}$ is finite over its center, $A(K,I)$ is finite over its center for any choice of $K$ and $I$ (see Lemmas~\ref{lem:956} and~\ref{lem:AunivPI}).  Also, there are practical conditions on $K$ and $I$ which ensure that $H$ acts inner faithfully on $A(K,I)$.  For example, if the sum $K\oplus V$ generates $\operatorname{Rep}(G)$ and the composition $V\to A_\mrm{univ}\to A(K,I)$ is injective then the $H$-action on $A(K,I)$ is inner faithful.
\par

In what follows we consider $H$-module structures on a given algebra $Q$ which are induced by a $B(V)$-module structure in $\YD(G)$.  An additional $\YD(G)$-structure on an $H$-module algebra $Q$ consists only of a choice of an additional action of the character group $G^\vee$ on $Q$, which is compatible with the given $H$-action.

\begin{proposition}\label{prop:q_univ}
Suppose $H=B(V)\rtimes G$ acts inner faithfully on a central simple division algebra $Q$.  Then
\begin{enumerate}
\item $Q$ admits an $H$-module algebra map $f:A_\mrm{univ}\to Q$ so that $x_i\cdot a=[f(c_i),a]_\mrm{sk}$ for each $x_i\in \Prim(H)'$ and $a\in Q$.
\item $Q$ contains an $H$-division subalgebra of the form $Q(K',I')$ for some pre-faithful pair $(K',I')$.
\item If the $H$-action on $Q$ is induced by a $B(V)$-module algebra structure in $\YD(G)$, then $Q$ contains an $H$-division subalgebra $Q'$ over which $Q$ is a finite module, and which admits an embedding $Q'\to Q(K,I)$ into a division algebra associated to a faithful pair.  In particular, the existence of such $Q$ impies the existence of a faithful pair for $H$.
\end{enumerate}
\end{proposition}

\begin{proof}
(1) By Lemma~\ref{lem:inn_der} the $x_i$ act on $Q$ as skew derivations
\[
x_i\cdot a=[c'_i,a]_\mrm{sk}=c'_ia-(g_i\cdot a)c'_i
\]
for some $c'_i\in Q$ of $G^\vee$-degree $\chi_i$.  (Here $(g_i,\chi_i)$ denotes the $G\times G^\vee$-degree of $x_i$ in $B(V)$.)  We claim that the assignment $f(c_i)=c'_i$ provides the necessary map of (1).  Indeed, the corresponding map $F:TV\to Q$, $F(x_i)=c'_i$ is a well-defined $TV\rtimes G$-module map, and factors through $A_\mrm{univ}$ as any relation $r$ for $B(V)$ is such that $F(r\cdot a)=r\cdot F(a)=0$.  Whence there is a well-defined $G$-algebra map $f:A_\mrm{univ}\to Q$, $f(c_i)=c'_i$, which commutes with the skew derivations $x_i\cdot-$, and is therefore a map of $H$-module algebras.
\par

(2) Take $K'$ to be a $G$-subfield in $Q$ which is contained in the $B(V)$-invariants, and which contains $Z(Q)^H$.  By Proposition~\ref{prop:735} $Q$ is finite over $K'$.  The $B(V)$-invariance of $K'$ tells us that all the $c'_i\in Q$, from (1), skew commute with $K'$.  Hence the map $f$ of (1) extends to $f':K'\underline{\ot} A_\mrm{univ}\to Q$.  Take $I'=\ker(f')$ to obtain the desired pre-faithful pair.
\par

(3) Via the Yetter-Drinfeld structure on $Q$, we may take each $c'_i\in Q$ of the appropriate $G\times G^\vee$-degree $(g_i,\chi_i)$.  The map $A_\mrm{univ}\to Q$ is then a map in $\YD(G)$, and inner faithfulness ensures that the composite $V\to A_\mrm{univ}\to Q$ is injective.  (Otherwise homogeneous elements in the kernel would act trivially on $Q$.)
\par

Take $Q'=Q(K',I')$ with $K'$ and $I'$ as in (2), and let $S=\operatorname{Sym}(W)$ where $W$ is a (finite-dimensional) $G$-representation such that $W\oplus Q'$ generates $\mrm{Rep}(G)$ as a tensor category.  If we take $S$ as a trivial $G$-comodule, the diagonal $H$-action on the tensor product $S\underline{\ot}Q'$ gives it an $H$-module algebra structure.  This algebra is a domain which is finite over its center, and so we take the ring of fractions to get a central simple algebra $Q''=\mrm{Frac}(S\underline{\ot}Q')$ on which $H$-acts inner faithfully.  If we take $K$ to be the image of the $G$-algebra $\Frac(S\ot K')$ in $Q''$, and $I$ the kernel of the map $K\underline{\ot}A_\mrm{univ}\to Q''$, then we see $Q''=Q(K,I)$.
\end{proof}

\begin{remark}
We have a faithful braided functor $\YD(G)\to \YD(G\times G^\vee)$ so that Hopf algebras in $\YD(G)$ are sent to Hopf algebras in $\YD(G\times G^\vee)$, and an extension of an $H$-action on $Q$ to a $B(V)$-action in $\YD(G)$ is equivalent to an action of the pointed algebra $B(V)\rtimes (G\times G^\vee)$ on $Q$.  So, in terms of the general question of (non-)existence of actions of pointed, coradically graded, Hopf algebras on central division algebras, one may deal only with actions of Nichols algebras in Yetter-Drinfeld categories.
\par

In particular, the non-existence of a faithful pair $(K,I)$ for a particularly pathological braided vector space $V$ in some $\YD(G)$ would provide a negative resolution to~\cite[Question 1.1]{cuadraetingof16}.  One could also attempt to approach actions on quantum tori~\cite[Conjecture 0.1]{artamonov05} via $A_\mrm{univ}$.
\end{remark}

Proposition~\ref{prop:q_univ} is, of course, why we refer to $A_\mrm{univ}$ as the universal algebra for $H$.

\subsection{A non-Cartan example}
\label{sect:ane}

We provide a small example to illustrate the manner in which $A_\mrm{univ}$ can be employed to obtain results outside of Cartan type.  Consider $V_2=\mbb{C}\{x_1,x_2\}$ the $2$-dimensional braided vector space with braiding matrix $
[q_{ij}]=\left[\begin{array}{cc}
-1 & \sqrt{-1}\\
-1 & \sqrt{-1}
\end{array}\right]$.  We take $V_2$ as an object in $\YD(\mbb{Z}/4\mbb{Z})$ with each of the $x_i$ homogeneous of degree $g$, where $g$ generates $\mbb{Z}/4\mbb{Z}$, and $g\cdot x_1=-x_1$, $g\cdot x_2=\sqrt{-1}x_2$.  Note that $V_2$ is a faithful $\mbb{Z}/4\mbb{Z}$-representation, and that $V_2$ is not of Cartan type, as $q_{12}q_{21}=-\sqrt{-1}$ is not in the orbit of $q_{11}=-1$.
\par

By~\cite{nichols78} (see also \cite[Remark 2.13]{grana00}),
the Nichols algebra $R=B(V_2)$ has relations
\begin{equation}\label{eq:78}
x_1^2=0,\ \ x_2^4=0,\ \ \mrm{ad}_\mrm{sk}(x_1)^2(x_2)=0,\ \ \mrm{ad}_\mrm{sk}(x_2)^2(x_1)=0.
\end{equation}
One can check directly, or use the fact that $x_1^2$ is primitive, to see that the relation $x_1^2=0$ implies the relation $\mrm{ad}_\mrm{sk}(x_1)^2(x_2)=0$.  Hence we have the minimal presentation
\[
B(V_2)=\mbb{C}\langle x_1,x_2\rangle/(x_1^2, x_2^4, \mrm{ad}_\mrm{sk}(x_2)^2(x_1)).
\]

One sees that each of the minimal relations for $B(V_2)$ is primitive in the tensor algebra $TV$ (see~\cite{andruskiewitschschneider02}).  Hence the universal algebra in this case has relations given by skew commutators
\[
A_\mrm{univ}(V_2)=\mbb{C}\langle c_1,c_2\rangle/([c_1^2,c_2]_\mrm{sk},[c_2^4,c_1]_\mrm{sk}, [\mrm{ad}_\mrm{sk}(c_2)^2(c_1),c_i]_\mrm{sk}).
\]
One checks directly that in the quotient algebra $\mbb{C}_i[c_1,c_2]=\mbb{C}\langle c_1,c_2\rangle/([c_1,c_2]_\mrm{sk})$ we have 
\[
[c_1^2,c_2]_\mrm{sk}=[c_2^4,c_1]_\mrm{sk}=0\ \ \text{and}\ \ \mrm{ad}_\mrm{sk}(c_2)^2(c_1)=0,
\]
which implies $[\mrm{ad}_\mrm{sk}(c_2)^2(c_1),c_i]_\mrm{sk}=0$.  Hence we have the obvious quotient $\pi:A_\mrm{univ}(V_2)\to \mbb{C}_i[c_1,c_2]$.  The pair $(\mbb{C},\mrm{ker}(\pi))$ is faithful, and so we produce a central simple division algebra
\[
Q(\mbb{C},\mrm{ker}(\pi))=\Frac(\mbb{C}_i[c_1,c_2])
\]
on which the non-Cartan type graded Hopf algebra $H=B(V_2)\rtimes \mbb{Z}/4\mbb{Z}$ acts inner faithfully.

\bibliographystyle{abbrv}
%\bibliography{D}

\end{document}